\documentclass[12pt]{article}
\usepackage{amsthm}
\usepackage{amsmath}
\usepackage{amssymb}
\usepackage{latexsym}
\usepackage{amsfonts}
\usepackage{color}
\usepackage{mathrsfs}
\usepackage{epsfig}
\usepackage{cite}
\usepackage{hyperref}
\newtheorem{theorem}{\bf Theorem}[section]
\newtheorem{lemma}[theorem]{\bf Lemma}

\numberwithin{equation}{section}
\numberwithin{theorem}{section}
\begin{document}
\title{Padovan numbers which are concatenations of three repdigits}
\author{Kisan Bhoi and Prasanta Kumar Ray}
\date{}
\maketitle
\begin{abstract}
Padovan sequence is a ternary recurrent sequence defined by the recurrence relation $P_{n+3}=P_{n+1}+P_{n}$ with initial terms $P_{0}=P_{1}=P_{2}=1.$ In this study it is shown that $114,151,200,265,351,465,616,816,3329,4410,7739,922111$ are the only Padovan numbers which are concatenations of three repdigits.
\end{abstract}
\section{Introduction}
A natural number $N$ is called a repdigit if it has only one distict digit in its decimal expansion. That is, $N$ is of the form $N=a\left(\frac{10^{m}-1}{9}\right)=\overline{\underbrace{ a \ldots a}_{m~ \text{times}}}$ for $0\leq a\leq 9$ and $m\geq1.$ A natural number of the form
\begin{equation*}
  N=\overline{\underbrace{a\ldots a}_{l~\text{times}}\underbrace{b\ldots b}_{m~\text{times}}\underbrace{c\ldots c}_{k~\text{times}}}
\end{equation*} can be viewed as concatenation of three repdigits where $a,l,m,k\geq1$ and $0\leq a,b,c\leq 9.$
Let $\{P_{n}\}_{n\geq0}$ be the Padovan sequence given by
\begin{equation*}
P_{n+3}=P_{n+1}+P_{n} ~\text{for all}~ n\geq1,
\end{equation*}
with $P_{0}=P_{1}=P_{2}=1.$ The first few terms of this sequence are
\begin{equation*}
  1, 1, 1, 2, 2, 3, 4, 5, 7, 9, 12, 16, 21, 28, 37, 49, 65, 86, 114, 151, 200, 265, 351,\ldots.
\end{equation*}
In recent times, concatenation of repdigits in different linear recurrence sequences has been studied by many researchers. For example, Alahmadi et al. \cite{ALAHMADI} studied the problem of finding all Fibonacci numbers which are concatenations of two repdigits. They examined the same problem for $k$-generalised Fibonacci sequence in \cite{ALAHMADI1}. Ddamulira searched all Padovan and Tribonacci numbers which are concatenations of two repdigits in \cite{Ddmulir}, \cite{Ddamulira}, respectively. Batte et al. \cite{Batte} proved that the only Perrin numbers which are concatenations of two repdigits are $\{10, 12, 17, 22, 29, 39, 51, 68, 90, 119, 277, 644\}$. In \cite{RAYGURU}, Rayguru and Panda showed that $35$ is the only balancing number which is concatenation
of two repdigits. Recently, Siar and Keskin \cite{Siar} found that $12, 13, 29, 33, 34, 70, 84, 88, 89, 228$, and $233$
are the only $k$-generalized Pell numbers, which are concatenations of two repdigits with at least two digits. Erduvan and Keskin \cite{Erduvan}, in their work obtained Lucas numbers which are concatenations of two repdigits. In \cite{Keskin}, the same authors looked into Lucas numbers which are concatenations of three repdigits. They proved that the only Lucas numbers which are concatenations of three repdigits are $123, 199, 322, 521, 843, 2207, 5778$. For the proof, they used lower bounds for linear forms in logarithms and Baker-Davenport reduction method due to Dujella and Peth\H o \cite{Dujella}.
In this paper, we find all Padovan numbers that are concatenations of three repdigits. More precisely, we completely solve the Diophantine equation
\begin{eqnarray*}
  P_{n}&=&\overline{\underbrace{a\ldots a}_{l~\text{times}}\underbrace{b\ldots b}_{m~\text{times}} \underbrace{c\ldots c}_{k~\text{times}}}\\
   &=&\overline{\underbrace{a\ldots a}_{l~\text{times}}}\cdot 10^{m+k}+\overline{\underbrace{b\ldots b}_{m~\text{times}}}\cdot 10^{k}+\overline{\underbrace{c\ldots c}_{k~\text{times}}} \\
   &=&a\left(\frac{10^{l}-1}{9}\right)10^{m+k}+b\left(\frac{10^{m}-1}{9}\right)10^{k}+c\left(\frac{10^{k}-1}{9}\right),
\end{eqnarray*} that is,
\begin{equation}\label{eq:1}
P_{n}=\frac{1}{9}\left(a 10^{l+m+k}-(a-b)10^{m+k}-(b-c)10^{k}-c\right).
\end{equation}
Our main result is the following.
\begin{theorem}\label{Th:1.}
  The only Padovan numbers that are concatenations of three repdigits are $114,151,200,265,351,465,616,816,3329,4410,7739,922111.$
\end{theorem}
For the proof of Theorem \ref{Th:1.}, we run a program in \textit{Mathemtica} and search all solutions of \eqref{eq:1} with $n\leq560$. Then, we take $n>560$ and write \eqref{eq:1} in three different ways to get three linear forms in logarithms. We apply lower bounds for linear forms in logarithms to obtain an upper bound on $n.$ To reduce the bound, we use Baker-Davenport reduction method de Weger \cite{Weger}. As a result, we get $n\leq 546,$ a contradiction. In this way, we complete the proof of our main result.

We need some preliminary results which are discussed in the following section.
\section{Preliminary Results}
\subsection{Some properties of Padovan sequence}
The characteristic equation of $\{P_{n}\}_{n\geq0}$ is $f(x)=x^{3}-x-1=0$ which has one real root $\alpha$ and two complex conjugate roots $\beta$ and $\gamma$ given by
\begin{equation*}
\alpha=\frac{r_{1}+r_{2}}{6},~~\beta=\bar{\gamma}=\frac{-r_{1}-r_{2}+i\sqrt{3}(r_{1}-r_{2})} {12},
\end{equation*} where
$r_{1}=\sqrt[3]{108+12\sqrt{69}}$ and $r_{2}=\sqrt[3]{108-12\sqrt{69}}$. Binet's formula for Padovan sequence is given by
\begin{equation}\label{eq:3}
  P_{n}=C_{\alpha}\alpha^{n}+C_{\beta}\beta^{n}+C_{\gamma}\gamma^{n},
\end{equation}
where
\begin{align*}
    &C_{\alpha}=\frac{(1-\beta)(1-\gamma)}{(\alpha-\beta)(\alpha-\gamma)}=\frac{\alpha+1}{-\alpha^2+3\alpha+1},\\
    &C_{\beta}=\frac{(1-\alpha)(1-\gamma)}{(\beta-\alpha)(\beta-\gamma)}=\frac{\beta+1}{-\beta^2+3\beta+1},\\ &C_{\gamma}=\frac{(1-\alpha)(1-\beta)}{(\gamma-\alpha)(\gamma-\beta)}=\frac{\gamma+1}{-\gamma^2+3\gamma+1}.
\end{align*}
The minimal polynomial of $C_{\alpha}$ over $\mathbb{Z}$ is $23x^{3}-23x^{2}+6x-1$ with all its zeros of unit modulus.
Numerically, we have the following estimates:
\begin{align*}
 1.32<\alpha<1.33,\\
 0.86<\vert\beta\rvert=\vert\gamma\vert =\alpha^{-1/2}<0.87,\\
 0.72<C_{\alpha}<0.73,\\
 0.24<\vert C_{\beta}\vert=\vert C_{\gamma}\vert<0.25.
 \end{align*}
 Let $e(n):=P_{n}-C_{\alpha}\alpha^{n}=C_{\beta}\beta^{n}+C_{\gamma}\gamma^{n}.$ Then $\vert e(n)\vert<\frac{1}{\alpha^{n/2}}$ for all $n\geq1.$
 Using induction, one can prove that
 \begin{equation}\label{eq:2}
 \alpha^{n-3}\leq P_{n}\leq \alpha^{n-1} ~\text{holds for all}~ n\geq 1.
 \end{equation}
\subsection{Linear forms in Logarithms}
Baker's theory plays an important role in reducing the bounds concerning linear forms in logarithms of algebraic numbers.
Let $\eta$ be an algebraic number with minimal primitive polynomial
\begin{align*}
     f\left(X\right) =a_{0}x^{d}+a_{1}x^{d-1}+\ldots+a_{d}= a_{0}\prod_{i=1}^{d}(X-\eta^{(i)})\in {\mathbb{Z}\left[X\right]},
\end{align*}
where  the leading coefficient $a_{0}>0$, and $\eta^{(i)}$'s are conjugates of $\eta$. Then, the \textit{logarithmic height} of $\eta$ is given by
\begin{align*}
    h(\eta)=\frac{1}{d}\left(\log a_{0}+\sum_{j=1}^{d}\text{max}\{0,\log\vert\eta^{(j)}\vert\}\right).
\end{align*}
The following are some properties of the logarithmic height function (see \cite[Property 3.3]{W}), which will be used to calculate the heights of the algebraic integers
\begin{gather*}
       h(\eta+\gamma)\leq h(\eta)+h(\gamma)+\log2,\\
    h(\eta\gamma^{\pm1})\leq h(\eta)+h(\gamma),\\
     h(\eta^{k})=|k|h(\eta), k\in \mathbb{Z}.
\end{gather*}
With these notations, Matveev (see \cite{Matveev} or \cite[Theorem 9.4]{Bugeaud1}) proved the following result.
 \begin{theorem}\label{th:2}
 Let $\eta_{1},\eta_{2},\ldots , \eta_{l}$ be positive real algebraic integers in a real algebraic number field $\mathbb{L}$  of degree $d_{\mathbb{L}}$ and $b_{1},b_{2},\ldots, b_{l}$ be non zero  integers. If
 $\Gamma=\prod_{i=1}^{l}\eta_{i}^{b_{i}}-1
$ is not zero, then
\begin{align*}
    \log \vert\Gamma\vert>-1.4\cdot30^{l+3}l^{4.5}d_{\mathbb{L}}^{2}(1+\log d_{\mathbb{L}})(1+\log D)A_{1}A_{2}\ldots A_{l},
    \end{align*} where $D=\text{max}\{\vert b_{1} \vert,\vert b_{2} \vert,\ldots, \vert b_{l} \vert\}$ and
$A_{1},A_{2},\ldots, A_{l}$ are positive real numbers such that
\begin{align*}
    A_{j}\geq \text{max}\{d_{\mathbb{L}}h\left( \eta_{j}\right),\vert\log\eta_{j}\vert, 0.16\}  ~~ \text{for}~   j=1,\ldots,l.
\end{align*}
\end{theorem}
\subsection{Reduction method}
The following Baker-Davenport reduction method due to de Weger \cite{Weger} is used to reduce the bounds of the variables which are too large.
Let $\vartheta_{1}, \vartheta_{2}, \beta \in \mathbb{R}$ be given and $ x_{1},x_{2} \in \mathbb{Z}$ be unknowns.
Suppose
\begin{align}\label{eq:T}
    \Lambda=\beta+x_{1}\vartheta_{1}+x_{2}\vartheta_{2}.
\end{align}
Set
$X=\max\{\vert x_{1} \vert, \vert x_{2} \vert\}$ and $ X_{0}, Y $ be positive.
Assume that
\begin{align}\label{eq:V}
    X\leq X_{0},
\end{align}
and
\begin{align}\label{eq:U}
    \vert\Lambda\vert < c \cdot\exp(-\delta\cdot Y),
\end{align}
where $c, \delta$ be positive constants.
When $\beta\neq0$ in \eqref{eq:T}, put $\vartheta=-\vartheta_{1}/\vartheta_{2}$ and $\psi=\beta/\vartheta_{2}$.
Then we have
\begin{align*}
    \frac{\Lambda}{\vartheta_{2}}=\psi-x_{1}\vartheta+x_{2}.
\end{align*}
Let $p/q$ be a convergent of $\vartheta$ with $q>X_{0}.$ For a real number $x,$ we let
$
\|x\|=\min\{\vert x-n \vert, n \in \mathbb{Z} \}$ be the distance from $x$ to the nearest integer. We have the following result.
\begin{lemma}(\cite[Lemma 3.3]{Weger})\label{le:2.3}
Suppose that
\begin{align*}
    \|q\psi\|>\frac{2X_{0}}{q}.
\end{align*}
Then, the solutions of \eqref{eq:U} and \eqref{eq:V} satisfy
\begin{align*}
     Y<\frac{1}{\delta}\log\left(\frac{q^{2}c}{\vert
     \vartheta_{2}\vert X_{0}}\right).
\end{align*}
\end{lemma}
We state and prove the following lemma which gives a relation between $n$ amd $l+m+k$ of \eqref{eq:1}.
\begin{lemma}
All solutions to the Diophantine equation \eqref{eq:1} satisfy
\begin{equation*}
  (l+m+k)\log10-3<n\log\alpha<(l+m+k)\log10+1.
\end{equation*}
\end{lemma}
\begin{proof}
From \eqref{eq:1} and \eqref{eq:2} we get
\begin{equation*}
   \alpha^{n-3}\leq P_{n}<10^{l+m+k}.
\end{equation*}
Taking logarithm on both sides, we have
\begin{equation*}
  (n-3)\log\alpha<(l+m+k)\log10,
\end{equation*}
which leads to
\begin{equation*}
  n\log\alpha<(l+m+k)\log10+3\log\alpha<(l+m+k)\log10+1.
\end{equation*}
On the other hand, for the lower bound, \eqref{eq:1} implies that
\begin{equation*}
  10^{l+m+k-1}<P_{n}\leq\alpha^{n-1}.
\end{equation*}
Taking logarithm on both sides, we get
\begin{equation*}
  (l+m+k-1)\log10<(n-1)\log\alpha,
\end{equation*} which leads to
\begin{equation*}
  (l+m+k-1)\log10-3<(l+m+k-1)\log10+\log\alpha<n\log\alpha.
\end{equation*}
\end{proof}
The following result will also be used in our proof.
\begin{lemma}\label{le:*}
If $\lvert e^{z}-1\rvert<y<\frac{1}{2}$ for real values of $z$ and $y$, then $\lvert z\rvert<2y$.
\end{lemma}
\begin{proof}
   If $z\geq0,$ we have $0\leq z\leq e^{z}-1=\lvert e^{z}-1\rvert<y.$ If $z<0,$ then $\lvert e^{z}-1\rvert <\frac{1}{2}.$ From this, we get $e^{\lvert z\rvert}<2,$ and therefore $0<\lvert z \rvert<e^{\lvert z\rvert}-1=e^{\lvert z \rvert} \lvert e^{z}-1\rvert<2y$. so, in both cases, we get $\lvert z\rvert<2y.$
\end{proof}
The following lemma will be used in our proof. It is seen in \cite[Lemma 7]{GUZMAN}.
\begin{lemma}\label{le:q}
Let $r\geq 1$ and $H>0$ be such that $H>(4r^{2})^r$ and $H>L/(\log L)^r$. Then
\begin{align*}
L<2^{r}H(\log H)^r.
\end{align*}
\end{lemma}
\section{Proof of the main result}
Using \textit{Mathematica}, we find all solutions to the Diophantine equation \eqref{eq:1} for $a\geq1$, $0\leq a,b,c\leq9$ and $1\leq n\leq 560$ that are listed in Theorem \ref{Th:1.}. Now, assume that $n>560.$ We ignore the cases $a=b\neq c$ and $a\neq b=c$ since these cases have been treated in \cite{Ddmulir}. We also ignore the case $a=b=c$ as it has been studied in \cite{Ddamulira*}.
Using \eqref{eq:1} and \eqref{eq:3}, we have
\begin{equation}\label{eq:4}
C_{\alpha}\alpha^{n}+C_{\beta}\beta^{n}+C_{\gamma}\gamma^{n}=\frac{1}{9}\left(a 10^{l+m+k}-(a-b)10^{m+k}-(b-c)10^{k}-c\right).
\end{equation}
We can write \eqref{eq:4}as
\begin{equation}\label{eq:5}
9C_{\alpha}\alpha^{n}-a 10^{l+m+k}=-9(C_{\beta}\beta^{n}+C_{\gamma}\gamma^{n})-(a-b)10^{m+k}-(b-c)10^{k}-c.
\end{equation}
Taking absolute values on both sides of \eqref{eq:5}, we have that
\begin{align*}
  \left\vert9C_{\alpha}\alpha^{n}-a 10^{l+m+k}\right \vert & \leq 9\left\vert ( C_{\beta}\beta^{n}+C_{\gamma}\gamma^{n})\right\vert+\lvert(a-b)\rvert10^{m+k}+\lvert(b-c)\rvert10^{k}+\lvert c\rvert\\
  &\leq9\alpha^{-n/2}+9\cdot10^{m+k}+9\cdot10^{k}+9\\
  &<11\cdot10^{m+k}
\end{align*}
Dividing both sides of the above inequality by $a10^{l+m+k}$ gives
\begin{equation}\label{eq:6}
  \left\vert \left(\frac{9C_{\alpha}}{a}\right)\alpha^{n}10^{-(l+m+k)}-1\right\vert<\frac{11\cdot10^{m+k}}{a10^{l+m+k}}<\frac{11}{10^{l}}.
\end{equation}
Let
\begin{equation*}
  \Gamma_{1}=\left(\frac{9C_{\alpha}}{a}\right)\alpha^{n}10^{-(l+m+k)}-1.
 \end{equation*}
 Note that $  \Gamma_{1}\neq 0.$ If $  \Gamma_{1}=0,$ then
 \begin{equation}\label{e:*}
   C_{\alpha}\alpha^{n}=\frac{a}{9}10^{l+m+k}.
 \end{equation}
Let $\sigma$ be the automorphism of the Galois group of the splitting field of $f(x)$ over $\mathbb{Q}$ defined by $\sigma(\alpha)=\beta,$ where $f(x)=x^{3}-x-1$ is the minimal polynomial of $\alpha$.
Applying $\sigma$ on both sides of \eqref{e:*}, and taking absolute values, we get
\begin{align*}
  \lvert C_{\beta}\beta^{n}\rvert= \left\vert\left(\frac{a}{9}10^{l+m+k}\right)\right\vert,
\end{align*} which is not possible since $\lvert C_{\beta}\beta^{n}\rvert <1$, whereas $\left\vert\left(\frac{a}{9}10^{l+m+k}\right)\right\vert>1.$
Therefore, $ \Gamma_{1}\neq 0.$
 Now, we are ready to apply Theorem \ref{th:2} with the following data:
\begin{align*}
    \eta_{1}= \frac{9C_{\alpha}}{a},~ \eta_{2}=\alpha,~  \eta_{3}=10,~ b_{1}=1,~ b_{2}=n,~ b_{3}=-(l+m+k),~ l=3,
\end{align*} with $d_{\mathbb{L}}=[\mathbb{Q}(\alpha):\mathbb Q] = 3.$ Since $l+m+k<n,$ take $D =n$. The heights of $\eta_{1}, \eta_{2}, \eta_{3}$ are calculated as follows:
\begin{equation*}
   h(\eta_{1})=h(9C_{\alpha}/a)\leq h(9)+h(C_{\alpha})+h(a)\leq 2\log9+\frac{\log23}{3}<5.44,
\end{equation*}
\begin{equation*}
h(\eta_{2})=h(\alpha)=\frac{\log\alpha}{3},~h(\eta_{3})=h(10)=\log 10.
\end{equation*}
Thus, we take
\begin{equation*}
 A _{1}=16.32,~~~A_{2}=\log\alpha,~~~~A_{3}=3\log 10.
\end{equation*}
Applying Theorem \ref{th:2} we find
\begin{align*}
 \log \vert\Gamma_{1}\vert &>-1.4\cdot30^{6}3^{4.5}3^2(1+\log3)(1+\log n)(16.32)(\log\alpha)(3\log 10)\\
 &>-8.58\cdot10^{13}\log(1+\log n).
\end{align*}
Comparing the above inequality with \eqref{eq:6} gives
\begin{align*}
l\log 10-\log 11<8.58\cdot10^{13}(1+\log n),
\end{align*}
which leads to
\begin{equation}\label{eq:14}
 l\log 10<8.59\cdot10^{13}(1+\log n).
\end{equation}
Rewriting \eqref{eq:4} in another way, we have
\begin{equation}\label{eq:7}
     9C_{\alpha}\alpha^{n}-(a 10^{l}-(a-b))10^{m+k}=-9(C_{\beta}\beta^{n}+C_{\gamma}\gamma^{n})-(b-c)10^{k}-c.
\end{equation}
Taking absolute values on both sides of \eqref{eq:7}, we obtain
 \begin{align*}
 \left\vert9C_{\alpha}\alpha^{n}-(a 10^{l}-(a-b))10^{m+k}\right\vert & \leq9\vert(C_{\beta}\beta^{n}+C_{\gamma}\gamma^{n})\vert+\lvert(b-c)\rvert10^{k}+\lvert c\rvert\\
 & \leq 9\alpha^{-n/2}+9\cdot10^{k}+9\\
 &<11\cdot10^{k}
 \end{align*}
  Dividing both sides by $(a 10^{l}-(a-b))10^{m+k}$, we obtain
 \begin{equation}\label{eq:8}
   \left\vert 1-\left(\frac{9C_{\alpha}}{a10^{l}-(a-b)}\right)\alpha^{n}10^{-(m+k)} \right\vert<\frac{11\cdot10^{k}}{(a 10^{l}-(a-b))10^{m+k}}<\frac{11}{10^{m}}.
 \end{equation}
 Put
 \begin{equation*}
    \Gamma_{2}= 1-\left(\frac{9C_{\alpha}}{a10^{l}-(a-b)}\right)\alpha^{n}10^{-(m+k)}.
 \end{equation*}
 By using similar reason as above we can show that $ \Gamma_{2}\neq0$.
 Now, we apply Theorem \ref{th:2} with the following parameters:
\begin{align*}
    \eta_{1}= \frac{9C_{\alpha}}{a10^{l}-(a-b)},~ \eta_{2}=\alpha,~  \eta_{3}=10,~ b_{1}=1,~ b_{2}=n,~ b_{3}=-(m+k),~ l=3,
\end{align*} with $d_{\mathbb{L}}=[\mathbb{Q}(\alpha):\mathbb Q] = 3.$ Since $m+k<n,$ take $D =n$.
 Using the properties of logarithmic height, we estimate $h(\eta_{1})$ as follows:
  \begin{align*}
h(\eta_{1}) & =h\left(\frac{9C_{\alpha}}{a10^{l}-(a-b)}\right) \leq h(9C_{\alpha})+h\left(a10^{l}-(a-b)\right)\\
 & \leq h(9)+h(C_{\alpha})+h(a10^{l})+h(a-b)+\log2\\
 &\leq  h(9)+h(C_{\alpha})+h(a)+lh(10)+h(a)+h(b)+2\log2\\
 &\leq 4\log9+l\log10+\frac{\log23}{3}+2\log2\\
 &<11.23+l\log10.
 \end{align*}
Hence from \eqref{eq:14}, we get
\begin{equation*}
 h(\eta_{1})<11.23+8.59\cdot10^{13}(1+\log n)<8.6\cdot10^{13}(1+\log n).
  \end{equation*}
Furthermore, $h(\eta_{2})=h(\alpha)=\frac{\log\alpha}{3}$ and $h(\eta_{3})=h(10)=\log 10.$
Thus, we take
\begin{equation*}
 A _{1}=25.8\cdot10^{13}(1+\log n),~~~A_{2}=\log\alpha,~~~~A_{3}=3\log 10.
\end{equation*}
By virtue of Theorem \ref{th:2} we have
\begin{align*}
 \log \vert\Gamma_{2}\vert &>-1.4\cdot30^{6}3^{4.5}3^2(1+\log3)(1+\log n)(25.8\cdot10^{13}(1+\log n))(\log\alpha)(3\log 10)\\
 &>-1.36\cdot10^{27}(1+\log n)^{2}.
\end{align*}
Comparing the above inequality with \eqref{eq:8} gives
\begin{align*}
m\log 10-\log 11<1.36\cdot10^{27}(1+\log n)^{2},
\end{align*}
which leads to
\begin{equation}\label{eq:15}
 m\log 10<1.37\cdot10^{27}(1+\log n)^{2}.
\end{equation}
Again, rearranging \eqref{eq:4}, we have
 \begin{equation}\label{eq:9}
  9C_{\alpha}\alpha^{n}-(a 10^{l+m}-(a-b)10^{m}-(b-c))10^{k}=-9(C_{\beta}\beta^{n}+C_{\gamma}\gamma^{n})-c.
 \end{equation}
 Taking absolute values on both sides of \eqref{eq:9}, we obtain
  \begin{align*}
 \left \lvert 9C_{\alpha}\alpha^{n}-(a10^{l+m}-(a-b)10^{m}-(b-c))10^{k}\right\rvert&\leq9\vert(C_{\beta}\beta^{n}+C_{\gamma}\gamma^{n})\vert+\lvert c\rvert\\
 &\leq9\alpha^{n/2}+9<10.
 \end{align*}
 Dividing both sides by $9C_{\alpha}\alpha^{n}$, we obtain
 \begin{equation}\label{eq:10}
   \left\lvert 1-\left(\frac{a10^{l+m}-(a-b)10^{m}-(b-c)}{9C_{\alpha}}\right)\alpha^{-n}10^{k}\right\rvert<\frac{10}{9C_{\alpha}\alpha^{n}}<\frac{2.5}{\alpha^{n}} .
 \end{equation}
 Let
 \begin{equation*}
     \Gamma_{3}= 1-\left(\frac{a10^{l+m}-(a-b)10^{m}-(b-c)}{9C_{\alpha}}\right)\alpha^{-n}10^{k}.
 \end{equation*}
 By using similar arguments as before it is seen that $ \Gamma_{3}\neq0$.
 Now, we apply Theorem \ref{th:2} with the following parameters:
\begin{align*}
    \eta_{1}= \frac{a10^{l+m}-(a-b)10^{m}-(b-c)}{9C_{\alpha}},~ \eta_{2}=\alpha,~  \eta_{3}=10,~ b_{1}=1,~ b_{2}=-n,~ b_{3}=k,~ l=3,
\end{align*} with $d_{\mathbb{L}}=[\mathbb{Q}(\alpha):\mathbb Q] = 3.$ Since $k<n,$ take $D =n$.
 Using the properties of logarithmic height, we estimate $h(\eta_{1})$ as follows:
  \begin{align*}
h(\eta_{1}) & =h\left(\frac{a10^{l+m}-(a-b)10^{m}-(b-c)}{9C_{\alpha}}\right)\\
 &\leq h(a10^{l+m})+h((a-b)10^{m})+h(b-c)+h(9C_{\alpha})+2\log2\\
 & \leq h(a)+h(10^{l+m})+h(a-b)+h(10^{m})+h(b-c)+h(9)+h(C_{\alpha})+2\log2\\
 &\leq  h(9)+(l+m)h(10)+h(a)+h(b)+mh(10)+h(b)+h(c)+h(9)+h(C_{\alpha})+4\log2\\
 &\leq 6\log9+(l+m)\log10+m\log10+\frac{\log23}{3}+4\log2\\
 &<17.1+ l\log10+2m\log10.
 \end{align*}
Hence from \eqref{eq:14} and \eqref{eq:15}, we get
\begin{equation*}
 h(\eta_{1})<17.1+8.59\cdot10^{13}(1+\log n)+2.74\cdot10^{27}(1+\log n)^{2}<2.76\cdot10^{27}(1+\log n)^{2}.
  \end{equation*}
Furthermore, $h(\eta_{2})=h(\alpha)=\frac{\log\alpha}{3}$ and $h(\eta_{3})=h(10)=\log 10.$
Thus, we take
\begin{equation*}
 A _{1}=8.28\cdot10^{27}(1+\log n)^{2},~~~A_{2}=\log\alpha,~~~~A_{3}=3\log 10.
\end{equation*}
Using Theorem \ref{th:2} we have
\begin{align*}
 \log \vert\Gamma_{3}\vert &>-1.4\cdot30^{6}3^{4.5}3^2(1+\log3)(1+\log n)\left(8.28\cdot10^{27}(1+\log n)^{2}\right)(\log\alpha)(3\log 10)\\
 &>-4.35\cdot10^{48}(1+\log n)^{3}.
\end{align*}
Comparing the above inequality with \eqref{eq:10} gives
\begin{align*}
n\log \alpha-\log 2.5<4.35\cdot10^{48}(1+\log n)^{3},
\end{align*}
which further implies
\begin{equation*}
 n<15.6\cdot10^{48}(1+\log n)^{3}.
\end{equation*}
With the notation of Lemma \ref{le:q}, we take $r=3$, $L=n$ and $H=15.6\cdot10^{48}$. Applying Lemma \ref{le:q}, we have
\begin{align*}
n&<2^{3}(15.6\cdot10^{48})(\log(15.6\cdot10^{48})^{3}\\
&<2\cdot10^{56}.
\end{align*}
Now, to reduce the bounds, let
\begin{equation*}
  \Lambda_{1}=(l+m+k)\log 10- n\log\alpha-\log\left(\frac{9C_{\alpha}}{a}\right) .
\end{equation*}
The inequality \eqref{eq:6} can be written as
    \begin{equation*}
     \lvert e^{-\Lambda_{1}}-1\rvert<\frac{11}{10^{l}}.
 \end{equation*}
  Notice that $-\Lambda_{1}\neq0$ as $e^{-\Lambda_{1}}-1=\Gamma_{1}\neq0.$  Assuming $l\geq 2$, the right-hand side in the above inequality is at most $\frac{1}{2}.$
  Using Lemma \ref{le:*} we obtain
\begin{equation*}
 0<\lvert \Lambda_{1}\rvert<\frac{22}{10^l},
\end{equation*}
 which implies that
 \begin{equation*}
    \left\lvert (l+m+k)\log 10- n\log\alpha-\log\left(\frac{9C_{\alpha}}{a}\right)\right\rvert<22\exp(-l\log10) .
 \end{equation*}
 We also have
\begin{align*}
     \frac{\Lambda_{1}}{\log10}=(l+m+k)-n\left(\frac{\log\alpha}{\log10}\right)-\frac{\log(9C_{\alpha}/a)}{\log10}.
\end{align*}
Thus, we take
\begin{align*}
     &c=22,~ \delta=\log10,~ X_{0}=2\cdot10^{56},~ \psi=-\frac{\log(9C_{\alpha}/a)}{\log10}\\
     &\vartheta=\frac{\log\alpha}{\log10},~\vartheta_{1}=-\log\alpha,~\vartheta_{2}=\log10,~\beta=-\log\left(\frac{9C_{\alpha}}{a}\right).
\end{align*}
We find that $\frac{p_{121}}{q_{121}}= \frac{17548495448098062534665425097069587403994668418373610729747}{143694755301644024543505669827725455817494147218758974051600}$ is the $121$th convergent of $\vartheta$ such that $q_{121}>X_{0}$ and it satisfies  $\|q\psi\|>\frac{2X_{0}}{q}$. Applying Lemma \ref{le:2.3}, we get $l<\frac{1}{\log10}\log\left(\frac{ 143694755301644024543505669827725455817494147218758974051600^{2}\cdot 22}{\log10\cdot2\cdot10^{56}}\right)\leq62$.
Now, let
\begin{equation*}
  \Lambda_{2}=(m+k)\log 10- n\log\alpha-\log\left(\frac{9C_{\alpha}}{a10^{l}-(a-b)}\right) .
\end{equation*}
From \eqref{eq:8}, it is seen that
\begin{equation*}
     \lvert e^{-\Lambda_{2}}-1\rvert<\frac{11}{10^{m}}.
\end{equation*}
Notice that $\Lambda_{2}\neq0$ as $e^{-\Lambda_{2}}-1=\Gamma_{2}\neq0.$  Assuming $m\geq 2$, the right-hand side in the above inequality is at most $\frac{1}{2}.$
  Using Lemma \ref{le:*} we obtain
\begin{equation*}
 0<\lvert \Lambda_{2}\rvert<\frac{22}{10^m},
\end{equation*}
 which implies that
 \begin{equation*}
    \left\lvert (m+k)\log 10- n\log\alpha-\log\left(\frac{9C_{\alpha}}{a10^{l}-(a-b)}\right)\right\rvert<22\exp(-m\log10) .
 \end{equation*}
  We also have
\begin{align*}
     \frac{\Lambda_{2}}{\log10}=(m+k)-n\left(\frac{\log\alpha}{\log10}\right)-\frac{\log(9C_{\alpha}/(a10^{l}-(a-b)))}{\log10}.
\end{align*}
We choose
\begin{align*}
     &c=22,~ \delta=\log10,~ X_{0}=2\cdot10^{56},~ \psi=-\frac{\log(9C_{\alpha}/(a10^{l}-(a-b)))}{\log10}\\
     &\vartheta=\frac{\log\alpha}{\log10},~\vartheta_{1}=-\log\alpha,~\vartheta_{2}=\log10,~\beta=-\log\left(\frac{9C_{\alpha}}{a10^{l}-(a-b)}\right).
\end{align*}
We find that $\frac{p_{124}}{q_{124}}= \frac{627183116477915636965979019027840650348169004436382793310275}{5135649646898035023105055510310316619786462973990152740408498}$ is the $124$th convergent of $\vartheta$ such that $q_{124}>X_{0}$ and it satisfies  $\|q\psi\|>\frac{2X_{0}}{q}$. Using Lemma \ref{le:2.3}, we get $m\leq66$.
Now, let
\begin{equation*}
  \Lambda_{3}=k\log 10- n\log\alpha+\log\left(\frac{a10^{l+m}-(a-b)10^{m}-(b-c)}{9C_{\alpha}}\right).
\end{equation*}
From \eqref{eq:10}, it is seen that
\begin{equation*}
     \lvert e^{\Lambda_{3}}-1\rvert<\frac{2.5}{\alpha^{n}}.
\end{equation*}

 Notice that $\Lambda_{3}\neq0$ as $e^{\Lambda_{3}}-1=\Gamma_{3}\neq0.$  Since $n>560$, the right-hand side in the above inequality is at most $\frac{1}{2}.$
Using Lemma \ref{le:*} we obtain
\begin{equation*}
 0<\lvert \Lambda_{3}\rvert<\frac{5}{\alpha^{n}},
\end{equation*}
 which implies that
 \begin{equation*}
    \left\lvert k\log 10- n\log\alpha+\log\left(\frac{a10^{l+m}-(a-b)10^{m}+(b-c)}{9C_{\alpha}}\right)\right\rvert<5\exp(-n\log\alpha) .
 \end{equation*}
  We also have
\begin{align*}
     \frac{\Lambda_{3}}{\log10}=k-n\left(\frac{\log\alpha}{\log10}\right)+\frac{\log((a10^{l+m}-(a-b)10^{m}+(b-c))/9C_{\alpha})}{\log10}.
\end{align*}
Thus, we take
\begin{align*}
     &c=5,~ \delta=\log\alpha,~ X_{0}=2\cdot10^{56},~ \psi=\frac{\log((a10^{l+m}-(a-b)10^{m}+(b-c))/9C_{\alpha})}{\log10}\\
     &\vartheta=\frac{\log\alpha}{\log10},~\vartheta_{1}=-\log\alpha,~\vartheta_{2}=\log10,~\beta=\log\left(\frac{a10^{l+m}-(a-b)10^{m}+(b-c)}{9C_{\alpha}}\right).
\end{align*}
We find that $\frac{p_{125}}{q_{125}}= \frac{2616187734565998164447356816798847641018036629887990573914437}{21422489321292086600361648904597606825079844749495429580710675}$ is the $125$th convergent of $\vartheta$ such that $q_{125}>X_{0}$ and it satisfies  $\|q\psi\|>\frac{2X_{0}}{q}$. Applying Lemma \ref{le:2.3}, we get $n\leq546$, which is a contradiction to our assumption that $n>560$. Hence, the theorem is proved.

\begin{flushright}
\begin{tabular}{l}
\textit{Department of Mathematics}\\
\textit{Sambalpur University, Jyoti Vihar, Burla, India}\\
\textit{kisanbhoi.95@suniv.ac.in} \\
\end{tabular}
\end{flushright}
\begin{flushright}
\begin{tabular}{l}
\textit{Department of Mathematics}\\
\textit{Sambalpur University, Jyoti Vihar, Burla, India}\\
\textit{prasantamath@suniv.ac.in} \\
\end{tabular}
\end{flushright}
\end{document}